\documentclass[12pt]{article}
\usepackage{geometry}
\usepackage{color}
\usepackage{amsmath, amsthm, amsfonts,amssymb}
\newtheorem{thm}{Theorem}[section]
\newtheorem{lem}{Lemma}[section]

\newtheorem{cor}{Corollary}[section]
\newtheorem{remark}{Remark}[section]

\newcommand{\ve}{\varepsilon}
\newcommand{\1}{1\!\!\,{\rm I}}

\title{
	Random integral operators related to the point processes \\
}
\author{A. A. Dorogovtsev\footnote{ Institute of Mathematics,
		National Academy of Sciences of Ukraine
		Tereshchenkivska Str. 3, Kiev 01601, Ukraine,  \textit{andrey.dorogovtsev@gmail.com}} \;and Ia. A. Korenovska\footnote{ Institute of Mathematics,
		National Academy of Sciences of Ukraine
		Tereshchenkivska Str. 3, Kiev 01601, Ukraine, \textit{iaroslava.korenovska@gmail.com}}}

\begin{document}
	\maketitle
	
	\begin{abstract}
	In the article we study properties of the random integral operator in $L_2(
	\mathbb{R})$ whose kernel is obtained as a convolution of Gaussian density with a stationary point process.
	\end{abstract}

\maketitle
	
	\section{Introduction}
	Let $\Theta$ be a stationary point process on the real line \cite{1}. In this paper we consider integral operators in $L_2(\mathbb{R})$ with the kernel 
	\begin{equation}
	\label{eq1}
		 k(u,v)=\sum_{\theta\in\Theta}p(u-\theta)p(v-\theta),
	\end{equation}
	where $p$ is some square-integrable function. The necessity in the investigation of such random kernels arises in the theory of stochastic flows. Namely, in the articles \cite{2,3} the strong random operators related to an Arratia flow \cite{4} were introduced. If  $\left\{x(u,t), u\in\mathbb{R},t\ge0\right\}$   is an Arratia flow then for every $f\in L_2(\mathbb{R})$ and $t>0\;$ $ T_tf(u)=f(x(u,t)),\;u\in\mathbb{R},$ is a random element in $L_2(\mathbb{R})$. It was proved in \cite{2} that $T_t$ is a strong random operator in Skorokhod sense \cite{7} but it is not a bounded random operator \cite{3}. Since it is known that the map $ x(\cdot,t):\mathbb{R}\to\mathbb{R}$ is a step function with probability one then for any function $f$ with a bounded support $L_2 (\mathbb{R})$-norm of $T_tf$ equals to zero with positive probability. To avoid such situation one can consider $f\ast p_{\ve}$, where $p_{\ve}$ is a density of normal distribution with zero mean and variance $\ve$. Then, due to the change of variable formula for an Arratia flow \cite{3}, one can obtain 
	\begin{equation}
	\label{eq2}
	\int_{\mathbb{R}}T_t(f\ast p_{\ve})(u)^2du=\sum_{\theta:\;\Delta y(\theta,t)>0}\Delta y(\theta,t)\int_{\mathbb{R}}\int_{\mathbb{R}}p_{\ve}(v_1-\theta)p_{\ve}(v_2-\theta)f(v_1)f(v_2)dv_1dv_2,
	\end{equation}
	where $\left\{y(u,s), u\in\mathbb{R}, s\in [0;t]\right\}$ is a conjugated Arratia flow \cite{4}. In the right part of \eqref{eq2} one may see the quadratic form of the operator similar to \eqref{eq1}. Hence, the knowledge of the properties of \eqref{eq1} can help us in the investigation of random operators constructed from the stochastic flows. The article continues studying of characteristics of random operators from \cite{5,6}.
	
	\section{Shifts of Gaussian density along a point process }
	We will start with the following statement.
	
	\begin{thm}
		\label{lem1}
		Let $\Theta$ be a stationary ergodic point process on $\mathbb{R}$ \cite{1} and $ E|\Theta\cap\left[\left.0;1\right)\right.|<+\infty .$ Then there exists an event $\Omega_0$ of probability one such that for each $\omega\in \Omega_0$ a linear combinanations of the functions $\left\{ p_{\ve}(\cdot-\theta(\omega));\;\theta(\omega)\in\Theta(\omega)\right\}$ are dense in $L_2(\mathbb{R})$.
	\end{thm}
	\begin{proof}
		Lets break the proof into steps:
		\begin{lem}
			\label{lem2}
			Let $\Theta$ be a stationary ergodic point process on $\mathbb{R}$ with $ E|\Theta\cap\left[\left.0;1\right)\right.|<+\infty .$ Then with probability one
			$$
			\sum_{\theta\in\Theta}\frac{1}{|\theta|}=+\infty.
			$$
		\end{lem}
		\begin{proof}
			 Its sufficient to prove that 
			\begin{equation}
			\label{eq3}
				\sum_{\theta\in\Theta\cap\left[\left.1;+\infty\right)\right.}\frac{1}{\theta}=+\infty\;\mbox{ a.s.}
			\end{equation}
	
	Since $	\sum_{\theta\in\Theta\cap\left[\left.1;+\infty\right)\right.}\frac{1}{\theta}\ge
	\sum_{n=1}^{\infty}\frac{1}{n+1}|\Theta\cap \left[\left. n;n+1\right)\right. |$, then it is enough to show that for a sequence $ \xi_n=|\Theta\cap \left[\left. n;n+1\right)\right. |$ the series $ \sum_{n=1}^{\infty}\frac{1}{n}\xi_n$ diverges almost surely. One may note that  $ \left\{\xi_n\right\}_{n=0}^{\infty}$ is a stationary, ergodic, and $ E|\xi_0|<\infty$. Hence, due to ergodic theorem, for $ S_n=\sum_{k=0}^{n}\xi_n$ the following convergence holds
	\begin{equation}
	\label{eq4}
	\frac{1}{n}S_n\to E\xi_0,\; n\to\infty \quad\mbox{a.s.}
	\end{equation}
	Thus, with probability one $ C=\sup_{n\in\mathbb{N}}\frac{1}{n}S_n<+\infty$, and there exists $N\in\mathbb{N}$ such that for any $n\ge N$
	\begin{equation}
	\label{eq5}
	\frac{1}{n}S_n\ge\frac{E\xi_0}{2}.
	\end{equation}
Using this one can check that for any $m\in\mathbb{N}$
	\begin{equation}
	\label{eq60}
	\sum_{k=2}^{m}\frac{1}{k}\xi_k=\sum_{k=2}^{m}\left(\frac{S_k}{k}-\frac{S_{k-1}}{k-1}\right)+\sum_{k=2}^{m}\frac{S_{k-1}}{k(k-1)}
\ge 2C+\sum_{k=2}^{m}\frac{S_{k-1}}{k(k-1)}.
	\end{equation}
 Hence, by \eqref{eq5}, the series $ \sum_{k=2}^{+\infty}\frac{S_{k-1}}{k(k-1)}$ diverges, which, by \eqref{eq60}, proves the statement.
			\end{proof}
		\begin{cor}
			\label{cor1}
			Using Lemma \ref{lem2} and Muntz theorem one may check that there exists $\Omega_0$ of probability one such that for any $\omega\in \Omega_0$ and $0<a<b$ a linear combinations of the functions $\left\{u^{\theta(\omega)},\;\theta(\omega)\in\Theta(\omega)\right\}$ are dense in $L_2([a;b])$.
		\end{cor}
		\begin{cor}
			\label{cor2}
		 There exists $\Omega_0$ of probability one such that for any $\omega\in \Omega_0$ and $a<b$ a linear combinations of the functions $\left\{e^{\theta(\omega)u},\;\theta(\omega)\in\Theta(\omega)\right\}$ are dense in $L_2([a;b])$.
	\end{cor}
	\begin{proof}
		Denote by $ LS\left\{f_k, k=\overline{1,n}\right\}$ the linear span of $ f_1,\ldots,f_n$. Lets notice that for any $a<b$ and $f\in L_2([a;b])$ the following relations hold
			$$
			d\left(f,LS\left\{e^{\theta u},\;\theta\in\Theta\right\}\right)^2_{L_2([a;b])}=
			\inf_{c_{\theta}}\int_a^b\left(f(u)-\sum_{\theta\in\Theta}c_{\theta}e^{\theta u}\right)^2du=
			$$ 
			
			$$
			=\inf_{c_{\theta}}\int_{e^a}^{e^b}\left(f(\ln u)-\sum_{\theta\in\Theta}c_{\theta}u^{\theta}\right)^2\frac{du}{u}\le e^{-a}d\left(\tilde{f}, LS\left\{v^{\theta },\theta\in\Theta\right\}\right)^2_{L_2([e^a;e^b])},
		$$
		where the function $\tilde{f}(u)=f(\ln u)$ from $ L_2([e^a;e^b])$. 
		
		Thus, due to Corollary \ref{cor1}, with probability one for any $a<b$ and $f\in L_2([a;b])$ $$
		d\left(f,LS\left\{e^{\theta u},\;\theta\in\Theta\right\}\right)_{L_2([a;b])}=0.$$
		 \end{proof}
		 \begin{cor}
		 	\label{cor3}
		 	There exists $\Omega_0$ of probability one such that for any $\omega\in \Omega_0$ and $a<b$ a linear combinations of the functions $\left\{p_{\ve}(u-\theta(\omega)),\;\theta(\omega)\in\Theta(\omega)\right\}$ are dense in $L_2([a;b])$.
		 \end{cor}
		 \begin{proof}
		 	To prove this statement lets consider a fixed point $\tilde{\theta}\in\Theta $, and a linear bounded operator $B$ in $L_2([a;b])$ such that $\left(Bf\right)(u)=f(u)h(u)$, where
		 	$$
		 	h(u)=\frac{1}{\sqrt{2\pi\ve}}e^{-\frac{u^2}{2\ve}}e^{-\frac{\tilde{\theta}^2}{2\ve}}.
		 	$$
		 	Then for any $a<b$ and $f\in L_2([a;b])$
		 	$$
		 		d\left(f,LS\left\{p_{\ve}(u-\theta),\theta\in\Theta\right\}\right)^2_{L_2([a;b])}=
		 	$$
		 	$$
		 	= d\left(B\left(f(u)\sqrt{2\pi\ve}e^{\frac{u^2}{2\ve}}e^{\frac{\tilde{\theta}^2}{2\ve}}\right), LS\left\{B\left(e^{-\frac{\tilde{\theta}^2-\theta^2}{2\ve}}e^{\frac{\theta u}{\ve}}\right),\theta\in\Theta\right\}\right)^2_{L_2([a;b])}=
		 	$$
		 	$$
		 		= d\left(B\tilde{f}(u), LS\left\{Be^{\frac{\theta u}{\ve}},\theta\in\Theta\right\}\right)^2_{L_2([a;b])},
		 	$$
		 	where $ \tilde{f}(u)=f(u)\sqrt{2\pi\ve}e^{\frac{u^2}{2\ve}}e^{\frac{\tilde{\theta}^2}{2\ve}}$. Since $B$ is a bounded linear operator in $L_2([a;b]) $ then
		 	$$
		 	d\left(B\tilde{f}(u), LS\left\{Be^{\frac{\theta u}{\ve}},\theta\in\Theta\right\}\right)^2_{L_2([a;b])}\le
		 	\|B\|^2 d\left(\tilde{f}(u), LS\left\{e^{\frac{\theta u}{\ve}},\theta\in\Theta\right\}\right)^2_{L_2([a;b])},
		 	$$
		 	which, due to Corollary \ref{cor2}, equals to 0.
		 	\end{proof}
		To end the proof of the theorem it is enough, by Corollary \ref{cor3}, to note that for any  $f\in L_2(\mathbb{R})$ $$
		 	 d\left(f(u), LS\left\{p_{\ve}(u-\theta ),\theta\in\Theta\right\}\right)^2_{L_2(\mathbb{R})}=
		 	 $$
		 	 $$=
		 	 \lim_{m\to\infty}d\left(f(u)\1_{[-m;m]}(u), LS\left\{p_{\ve}(u-\theta ),\theta\in\Theta\right\}\right)^2_{L_2([-m;m]}.
		 	$$
		Consequently, with probability one the linear span of the functions $\left\{ p_{\ve}(\cdot-\theta);\;\theta\in\Theta\right\}$ is dense in $L_2(\mathbb{R})$. The theorem is proved.
		\end{proof}
		
			\section{Properties of the integral random operator}
		Now let us turn to the integral operator with the kernel \eqref{eq1}. Let $p_{\ve}$ be the same as before.
		
		\begin{lem}
			\label{lem3}
			For any $f\in L_2(\mathbb{R})$ and a stationary point process $\Theta$ with $E|\Theta\cap[0;1]|<+\infty$
			$$
			\sum_{\theta\in\Theta}\left(\int_{\mathbb{R}}f(u)p_{\ve}(u-\theta)du\right)^2<+\infty\;\;\mbox{a.s.}
			$$
			\begin{proof} Using Campbell's formula \cite{1} one can check that for every $f\in L_2(\mathbb{R})$ 
				$$
				E\sum_{\theta\in\Theta}\left(\int_{\mathbb{R}}f(u)p_{\ve}(u-\theta)du\right)^2\le
					E\sum_{\theta\in\Theta}\int_{\mathbb{R}}\int_{\mathbb{R}}|f(u)||f(v)|p_{\ve}(u-\theta)p_{\ve}(v-\theta)dudv=
				$$
				$$
		=	\int_{\mathbb{R}}\int_{\mathbb{R}}|f(u)||f(v)|	E\sum_{\theta\in\Theta}p_{\ve}(u-\theta)p_{\ve}(v-\theta)dudv=
			$$
			$$
			=	C\int_{\mathbb{R}}\int_{\mathbb{R}}|f(u)||f(v)|\int_{\mathbb{R}}p_{\ve}(u-t)p_{\ve}(v-t)dtdudv=
				$$
				$$
				=C\int_{\mathbb{R}}\int_{\mathbb{R}}|f(u)||f(v)|p_{2\ve}(u-v)dudv=C\int_{\mathbb{R}}h^2(\lambda)e^{-\ve\lambda^2}d\lambda\le C\int_{\mathbb{R}}|f(u)|^2du,
				$$
				where $C=E|\Theta\cap [0;1]|,$ and $h$ is the Fourier transform of $f\in L_2(\mathbb{R})$.
				\end{proof}
			\end{lem}
				\begin{remark}
					\label{rem1}
					It follows from the proof of Lemma \ref{lem3} that the following integral operator 
					$$
					Af(v)=\sum_{\theta\in\Theta}\int_{\mathbb{R}}f(u)p_{\ve}(u-\theta)du\cdot p_{\ve}(v-\theta)
					$$
					is well-defined and is a strong random operator in Skorokhod sense \cite{7}.
			\end{remark}	
					Next lemma shows that $A$ is not a bounded random operator in most interesting cases.
								
			\begin{lem}
				\label{lem4}
				Let $\Theta$ be an ergodic stationary point process such that
			$ essup |\Theta\cap[0;1]|=+\infty$. Then $A$ is not a bounded random operator.
				\begin{proof}
					It can be checked that under the condition on the process $\Theta$ with probability one there exists an increasing sequence of natural numbers $\left\{n_k;\;k\ge1\right\}$ such that
					$$
					\sup_{k\ge1}|\Theta\cap\left[\left.n_k;n_{k}+1\right)\right.|=+\infty.
					$$
					Consider the following sequence of functions from $L_2(\mathbb{R})$
					$$
					f_k=\1_{\left[\left.n_k;n_{k}+1\right)\right.},\quad k\ge1.
					$$
					Then 
					$$
					\|Af_k\|^2\ge\sum_{\theta\in\Theta\cap \left[\left.n_k;n_{k}+1\right)\right.}\left(\int_0^1p_{\ve}(v)dv\right)^2p_{\ve}(1)^2.
					$$
					Hence, $ \sup_k\|Af_k\|=+\infty,$  and lemma is proved.
					\end{proof}
			\end{lem}
		For a fixed interval $[a;b]$ lets denote by $Q_{a,b}$ the projection in $L_2(\mathbb{R})$ onto $ L_2([a;b])$. 
			
			\begin{remark}
				\label{rem2}
				For any $a,b\in\mathbb{R}$ the random operators $ AQ_{a,b}$, $ Q_{a,b}A$ are bounded. 
			\end{remark}
			\begin{proof}
			One can check, by H\"older inequality, that for any  $f,g\in L_2(\mathbb{R})$ 
			$$
			\left(AQ_{a,b}f,g\right)=\sum_{\theta\in\Theta}\int_{\mathbb{R}}g(v)p_{\ve}(v-\theta)\int_a^bf(u)p_{\ve}(u-\theta)du\le
			$$
				$$
			\le	2^{-\frac{1}{4}}(b-a)^{\frac{1}{2}}\|g\|_{L_2(\mathbb{R})}\|f\|_{L_2(\mathbb{R})}\sum_{\theta\in\Theta}\max_{u\in[a;b]}p_{\ve}(u-\theta).
				$$
			By Campbell's formula \cite{1},
			$$
			E\sum_{\theta\in\Theta}\max_{u\in[a;b]}p_{\ve}(u-\theta)=\int_{\mathbb{R}}\max_{u\in[a;b]}p_{\ve}(u-r)dr< +\infty.
			$$
		Thus, $\sum_{\theta\in\Theta}\max_{u\in[a;b]}p_{\ve}(u-\theta)<+\infty $ a.s., which proves the statement.
				\end{proof}
			\begin{lem}
				\label{lem5}
				For any $a,b\in\mathbb{R}$	with probability one the random operator 
				$
				A_{Q_{a,b}}=Q_{a,b}AQ_{a,b}
				$ 				is a nuclear.
				\begin{proof}
					To prove the statement lets estimate the nuclear norm of $Q_{a,b}AQ_{a,b}$. For any $ \theta\in\Theta$ denote by $e_{\theta}$ the function
					$$
					e_{\theta}=Q_{a,b}p_{\ve}(\cdot-\theta).
					$$
					Evidently, the operator $ e_{\theta}\otimes e_{\theta}$ is a nuclear, and its nuclear norm equals to $ \|e_{\theta}\|^2.$ Lets notice that 
					$$
					E\sum_{\theta\in\Theta}\|e_{\theta}\|^2=E\sum_{\theta\in\Theta}\int_{a}^bp_{\ve}(u-\theta)^2du=
					$$
					$$
					=C\int_a^b\int_{\mathbb{R}}p_{\ve}(u-v)^2dvdu<+\infty,
					$$
					where, as before, $ C=E|\Theta\cap [0;1]|.$ Its enough to note that 
					\begin{equation}
					\label{eq6}
					Q_{a,b}AQ_{a,b}=\sum_{\theta\in\Theta}e_{\theta}\otimes e_{\theta}.
					\end{equation}
					Lemma is proved.
					\end{proof}
			\end{lem}
			
			Due to the previous statement, the image $K$ of the unit ball in $L_2([a;b])$ under the operator $A_{Q_{a,b}}$ is a compact set with probability one. We obtain the following statement about asymptotic behavior of Kolmogorov width for the compact set $K$.
			\begin{thm}
				\label{lem6} Let $\Theta$ be an ergodic stationary point process. Then with probability one there exists $C>0$ such that
				$$
				d_n(K)= O\left(e^{-\frac{(Cn-b)^2}{\ve}}\vee e^{-\frac{(Cn+a)^2}{\ve}}\right),\; n\to\infty.
				$$
			\end{thm}
			\begin{proof}
			The representation \eqref{eq6} allows to estimate Kolmogorov widths of $K$.  Lets denote by $N_x$, $x>0$, the number of elements in the set $\Theta\cap[-x;x]$, and by $d_n$ the $n$-th Kolmogorov width of $K$. It follows from \eqref{eq6} that 
			\begin{equation}
			\label{eq7}
				d_{N_x}\le\sum_{\theta\in\Theta \setminus [-x;x]}\|e_{\theta}\|^2.
			\end{equation}
			Due to ergodic theorem,  $N_x\sim 2Cx$ when $x\to+\infty$. 
			
			To estimate the right part of \eqref{eq7} suppose that $ x>\max \{-a,b\}$, and consider the sum 
			$$
			\sum_{\theta\in\Theta, \theta>x}\|e_{\theta}\|^2\le\sum_{\theta\in\Theta,\theta>x}(b-a)p_{\ve}(\theta-b)^2.
			$$
			Denote by $ \xi_n=|\Theta\cap\left[\left.n;n+1\right)\right.|$. Then $\left\{\xi_n;n\ge1\right\}$ is a stationary ergodic sequence. For a natural $x$
			$$
			\sum_{\theta\in\Theta,\theta>x}p_{\ve}(\theta-b)^2\le
			\sum_{k=x}^{+\infty}p_{\ve}(x-b)^2\xi_k.
			$$
			For any $k\ge 1$ let $ S_k=\sum_{j=1}^{k}\xi_j.$ Since $		S_k\sim Ck,\;k\to\infty\mbox{ a.s.}, 		$ then, by Abel transform,	one can check that 
		$$
		\sum_{k=x}^{+\infty}p_{\ve}(k-b)^2\xi_k=-p_{\ve}(x-b)^2S_{x-1}+
		\sum_{k=x}^{+\infty}S_k(p_{\ve}(k-b)^2-p_{\ve}(k+1-b)^2)\sim
		$$
		$$
		\sim C\sum_{k=x}^{+\infty}(p_{\ve}(k-b)^2-p_{\ve}(k+1-b)^2)k,\;\;x\to+\infty.
		$$
		Lets notice that
		$$
		\sum_{k=x}^{+\infty}(p_{\ve}(k-b)^2-p_{\ve}(k+1-b)^2)k=\frac{1}{2\pi\ve}\sum_{k=x}^{+\infty}(1-e^{-\frac{(2k+1-2b)}{\ve}})e^{-\frac{(k-b)^2}{\ve}}k\sim \frac{1}{4\pi}e^{-\frac{(x-b)^2}{\ve}},
		$$
		and the statement is proved.
		\end{proof}
For any interval  $[a;b]$ $ A_{Q_{a,b}}$ is a bounded (nuclear) random operator. Despite this, when $ [a;b]$ increases to $ \mathbb{R}$, $ A_{Q_{a,b}}$ must converge to unbounded random operator $A$. Consequently, one can expect that the operator norm $\|A_{Q_{a,b}}\|$ will increase to infinity when $[a;b]$ increases to $\mathbb{R}$. Using the arguments from the proof of Lemma \ref{lem4} one can prove the following statement.

	\begin{thm}
		\label{thm1}
		Let $\Theta$ be a Poisson point process with intensity 1. Then
		$$
		\frac{\ln \ln n}{\ln n}\|A_{Q_{-n,n}}\|\to+\infty,\;n\to\infty\;\mbox{a.s.}
		$$
	\end{thm}
	\begin{proof}
		It follows from the proof of Lemma \ref{lem4} that 
		$$
		\|A_{Q_{-n,n}}\|\ge C\max_{\overline{1,n}}\xi_k,
		$$	
		where the random variables $\{\xi_n;n\ge1\}$ were introduced before. Now $\{\xi_n;n\ge1\}$ are independent random variables with poissonian distribution with intensity 1. Consequently,
		$$
		P\{\xi_1\ge m\}\sim \frac{e^{-1}}{m!},\;\;m\to+\infty.
		$$
		For any $R>0$ $ P\{\max_{k=\overline{1,n}}\xi_k\le m_nR\}=(1-P\{\xi_1>m_nR\})^n$. Thus, for $ m_n=\frac{\ln n}{\ln\ln n}$
		$$
		\frac{\max_{k=\overline{1,n}}\xi_k}{m_n}\to+\infty,\;n\to\infty\quad\mbox{a.s.,}
		$$
		and the theorem is proved.
		\end{proof}
		
		\par\bigskip\noindent
		{\bf Acknowledgment.} The first author acknowledges financial support from the Deutsche Forschungsgemeinschaft (DFG) within the project "Stochastic Calculus and Geometry of Stochastic Flows with Singular Interaction" for initiation of international collaboration between the Institute of Mathematics of the Friedrich-Schiller University Jena (Germany) and the Institute of Mathematics of the National Academy of Sciences of Ukraine, Kiev.


\begin{thebibliography}{99}
		\bibitem{1}
G.~ Last, M.~ Penrose. Lectures on the Poisson Processes. Cambridge University Press, 2016 (to appear).
		\bibitem{2}
A.~A.~Dorogovtsev.	Krylov-Veretennikov expansion for coalescing stochastic flows. // Commun. Stoch. Anal. --- 2012. --- \textbf{6}, No~3. --- P.~421-435.
		\bibitem{3}
			Ia.~A.~Korenovska. Properties of strong random operators generated by an Arratia flow (in Russian). // Ukr. Mat. Zh. - 2017. - 69, № 2. - pp. 157-172
			\bibitem{4}
R.~Arratia. Coalescing Brownian motions on the line. PhD Thesis, University of Wisconsin, Madison, 1979.
\bibitem{7}
A.V.Skorokhod. Random linear operators. D.Reidel Publishing Company, Dordrecht, Holland, 1983, 198 p.
\bibitem{5}
A.~A.~Dorogovtsev. Semigroups of finite-dimensional random projections. Lithuanian Mathematical Journal, v.51 (2011), No.3, p.p.330-341.
\bibitem{6}
I.~A.~Korenovska. Random maps and Kolmogorov widths. // Theory of Stochastic Processes. --- 2015. --- 20(36), No~1. --- P.~78-83.

	\end{thebibliography}
	\end{document}